\RequirePackage[l2tabu, orthodox]{nag}
\documentclass[a4paper,12pt]{amsart}
\usepackage{microtype}
\usepackage[charter]{mathdesign}

\makeatletter
\newcommand{\leqnomode}{\tagsleft@true\let\veqno\@@leqno}
\makeatother
\usepackage[T1]{fontenc}
\usepackage{tikz, tikz-cd, stmaryrd, amsmath, amsthm, amssymb,
hyperref, bbm, mathtools, mathrsfs}
\usepackage[all]{xy}

\usepackage[shortlabels]{enumitem}
\usetikzlibrary{arrows}
\usetikzlibrary{positioning}
\usepackage[utf8x]{inputenc}
\newcommand{\bb}{\textbf}

\newcommand{\ol}{\overline}
\newcommand{\mc}{\mathcal}

\newcommand{\mf}{\mathfrak}
\newcommand{\ms}{\mathscr}

\newcommand{\II}{\mathbb{I}}
\newcommand{\HH}{\mathbb{H}}
\newcommand{\ZZ}{\mathbb{Z}}

\newcommand{\RR}{\mathbb{R}}
\newcommand{\PP}{\mathbb{P}}

\newcommand{\FF}{\mathbb{F}}

\DeclareMathOperator{\tr}{tr}

\DeclareMathOperator{\Divv}{Div}

\DeclareMathOperator{\im}{im}
\DeclareMathOperator{\id}{id}

\DeclareMathOperator{\Span}{Span}

\DeclareMathOperator{\ord}{ord}

\DeclareMathOperator{\Hom}{Hom}

\DeclareMathOperator{\cha}{char}

\DeclareMathOperator{\Aut}{Aut}

\DeclareMathOperator{\Indec}{Indec}

\theoremstyle{plain}
\newtheorem{Theorem}{Theorem}[section]
\newtheorem*{mainthm}{Main Theorem}
\newtheorem{Remark}[Theorem]{Remark}
\newtheorem{Lemma}[Theorem]{Lemma}

\newtheorem{Proposition}[Theorem]{Proposition}

\theoremstyle{definition}

\numberwithin{equation}{section}
\begin{document}

\title[De Rham cohomology of covers]{The de Rham cohomology of covers\\ with cyclic $p$-Sylow subgroup}
\author[J. Garnek]{J\k{e}drzej Garnek}
\address{Faculty of Mathematics and Computer Science,
	Adam Mickiewicz University\\
	\quad ul. Uniwersytetu Pozna\'{n}skiego~4, \mbox{61-614} Poznan, Poland}
\email{jgarnek@amu.edu.pl}
\urladdr{http://jgarnek.faculty.wmi.amu.edu.pl/}

\author[A. Kontogeorgis]{Aristides Kontogeorgis}
\address{Department of Mathematics, National and Kapodistrian  University of Athens
	Panepistimioupolis, 15784 Athens, Greece}
\email{kontogar@math.uoa.gr}
\urladdr{http://users.uoa.gr/~kontogar}

\date{}  

\begin{abstract}
	Let $X$ be a smooth projective curve over a field $k$ with an action of a finite group $G$.
	A well-known result of Chevalley and Weil describes the $k[G]$-module structure of cohomologies of~$X$
	in the case when the characteristic of $k$ does not divide $\# G$. It is unlikely that such a formula can be derived in the
	general case, since the representation theory of groups with non-cyclic $p$-Sylow subgroups
	is wild in characteristic~$p$. The goal of this article is to show that when $G$ has a cyclic $p$-Sylow subgroup, the $G$-structure
	of the de Rham cohomology of $X$ is completely determined by the ramification data.
	In principle, this leads to new formulas in the spirit of Chevalley and Weil for such curves.
	We provide such an explicit description of the de Rham cohomology in the cases when $G = \ZZ/p^n$ and when the $p$-Sylow subgroup of $G$ is normal of order~$p$.
\end{abstract}

\maketitle
\bibliographystyle{plain}
\section{Introduction}
In the 1930s Chevalley and Weil gave an explicit description of the equivariant structure of the cohomology of a curve $X$ with an action of a finite group~$G$ over a field of characteristic~$0$ (cf. \cite{Chevalley_Weil_Uber_verhalten}, \cite{Ellingsrud_Lonsted_Equivariant_Lefschetz}). Their formula depends on the genus of the quotient curve $X/G$
and on the ramification data of the quotient morphism $X \to X/G$.
This result generalizes to the case $\cha k \nmid \#G$. However,
it is hard to expect such a formula for all finite groups.
Indeed, if $G$ is a group with a non-cyclic $p$-Sylow subgroup, the set of indecomposable $k[G]$-modules is infinite. If, moreover, $p > 2$ then the indecomposable $k[G]$-modules are considered impossible to classify (cf. \cite{Prest}). There are many results concerning the equivariant structure of cohomologies  for particular groups
(see e.g.~\cite{Valentini_Madan_Automorphisms} for the case of cyclic groups, \cite{WardMarques_HoloDiffs} for abelian groups, \cite{Bleher_Chinburg_Kontogeorgis_Galois_structure} for groups with a cyclic Sylow subgroup, or \cite{Bleher_Camacho_Holomorphic_differentials} for the Klein group) or curves (cf. \cite{Lusztig_Coxeter_orbits}, \cite{Dummigan_99}, \cite{Gross_Rigid_local_systems_Gm}, \cite{laurent_kock_drinfeld}). Also, one may expect that (at least in the case of $p$-groups) determining cohomologies comes down to Harbater--Katz--Gabber covers (cf. \cite{Garnek_p_gp_covers}, \cite{Garnek_p_gp_covers_ii}). However, there is no hope of obtaining a result similar to the one of Chevalley and Weil.\\

This brings attention to groups with a cyclic $p$-Sylow subgroup. For those, the set of 
equivalence classes of indecomposable modules is finite (cf. \cite{Higman}, \cite{Borevic_Faddeev}, \cite{Heller_Reiner_Reps_in_integers_I}). While their representation theory still seems a bit too complicated to derive a general formula for the cohomologies, 
the article~\cite{Bleher_Chinburg_Kontogeorgis_Galois_structure} shows that 
the $k[G]$-module structure of $H^0(X, \Omega_X)$ is determined by the genus of $X$ and ramification data (i.e. higher ramification groups and the fundamental characters of the ramification locus). The main result of this article is a similar statement for the de Rham cohomology.
\begin{mainthm}
	Suppose that $G$ is a group with a $p$-cyclic Sylow subgroup.
	Let $X$ be a curve with an action of~$G$ over a field $k$ of characteristic $p$.
	The $k[G]$-module structure of $H^1_{dR}(X)$ is uniquely determined by the  ramification data of the cover $X \to X/G$ and the genus of $X$.
\end{mainthm}
In order to prove Main Theorem, we show explicit formulas for the $G$-structure of $H^1_{dR}(X)$ when $G \cong \ZZ/p^n$ (cf. Theorem~\ref{thm:cyclic_de_rham})
and when $G$ has a normal Sylow subgroup of order~$p$ (cf. Theorem~\ref{thm:Zp_formula}).
In principle, the proof of Main Theorem could be made effective to give an explicit formula for $H^1_{dR}(X)$ for an arbitrary group with a cyclic $p$-Sylow subgroup.
This seems however really complicated for several reasons. Firstly, even though the indecomposable modules for groups
with a cyclic $p$-Sylow subgroup can be explicitly described (cf. \cite{Janusz_indecomposable_reps_cyclic_sylow}), this description is pretty 
involved. Secondly, already in the case when $G \cong \ZZ/p \rtimes_{\chi} C$, $p \nmid \# C$,
the explicit formula is quite long.\\

We elaborate now on the proof of Main Theorem. The first step is to provide an explicit formula in the case when $G = \ZZ/p^n$. This result is proven by applying induction on~$n$ twice: once for the curve $X$ with an action of $\ZZ/p^{n-1}$ and once for the curve $X'' := X/(\ZZ/p)$. 
The second step is to describe $H^1_{dR}(X)$ in case when $G \cong \ZZ/p \rtimes_{\chi} C$ and $p \nmid \# C$.
The proof in this case follows by analysing the $\ZZ/p$-invariants in the Hodge--de Rham exact sequence. 
Then we prove the result for groups of the form
$\ZZ/p^n \rtimes_{\chi} \ZZ/c$ by using induction similarly as in the first step. Finally, 
we use Conlon induction theorem to reduce Main Theorem to the case
when~$G$ is of the form $\ZZ/p^n \rtimes_{\chi} \ZZ/c$.\\

Note that if $p > 2$ and the $p$-Sylow subgroup of $G$ is not cyclic, the structure
of $H^1_{dR}(X)$ isn't determined uniquely by the ramification data. Indeed, see \cite{garnek_indecomposables} for a construction of $G$-covers with the same ramification data, but varying $k[G]$-module structure of $H^0(X, \Omega_X)$ and~$H^1_{dR}(X)$.
\subsection*{Outline of the paper}
In Section~\ref{sec:notation} we discuss notation and preliminaries. Section~\ref{sec:rep_theory} contains several results concerning modular representation theory.
In Section~\ref{sec:cyclic} we show the formula for the de Rham cohomology
of $\ZZ/p^n$-covers. The next Section is devoted to the proof of Main Theorem, assuming the result for groups
with a $p$-Sylow subgroup of order~$p$. The formula for such groups is proven in Section~\ref{sec:Zp-sylow}. In the final section
we provide an explicit example of curves with an action of $G := \ZZ/p \rtimes_{\chi} \ZZ/(m \cdot (p-1))$ and compute
the $G$-structure of their de Rham cohomologies.

\bigskip
\noindent {\bf Acknowledgements} 
 The research project is implemented in the framework of H.F.R.I Call “Basic research Financing Horizontal support of all Sciences” under the National Recovery and Resilience Plan “Greece 2.0” funded by the European Union Next Generation EU, H.F.R.I.  
Project Number: 14907.
\begin{center}
\includegraphics[width=0.6 \textwidth]{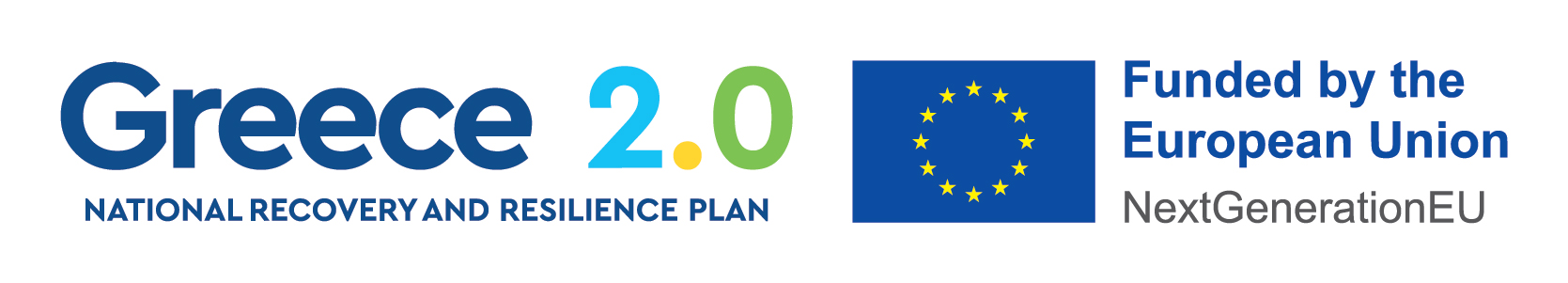}
\hskip 1cm
\includegraphics[width=0.3 \textwidth]{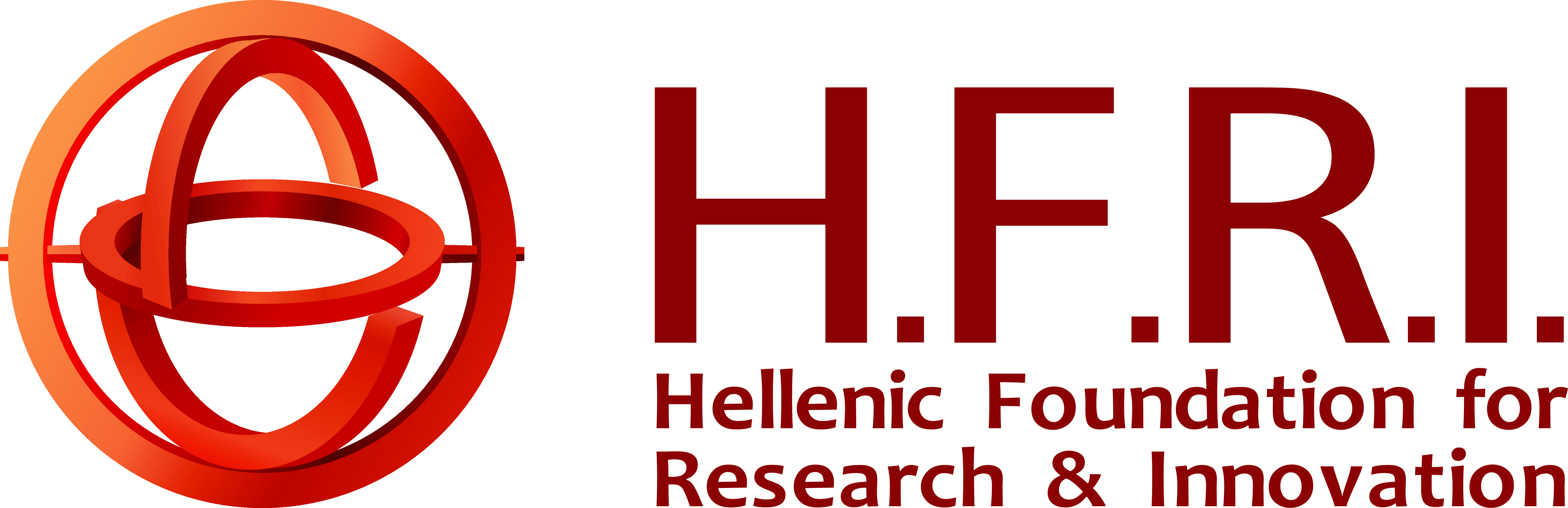}
\end{center}

\section{Notation and preliminaries} \label{sec:notation}
Assume that $\pi : X \to Y$ is a $G$-cover of smooth projective curves over a field $k$
of characteristic $p$.
Throughout the paper we will use the following notation for any $P \in X(\ol k)$:
\begin{itemize}
	\item $e_{X/Y, P}$ is the ramification index at $P$,
	
	\item $m_{X/Y, P} := \ord_p(e_{X/Y, P})$ is the maximal power of~$p$
	dividing the ramification index,
	
	\item $m_{X/Y} := \max \{ m_{X/Y, P} : P \in X(k) \}$,
	
	\item $u_{X/Y, P}^{(t)}$ (resp. $l_{X/Y, P}^{(t)}$) is the $t$-th upper (resp. lower) ramification jump
	at $P$ for $t \ge 1$,
	
	\item $u^{(0)}_{X/Y, P} := 1$ for any ramified point $P \in X(\ol k)$
	(note that this is not a standard convention),
	
	\item $u_{X/Y, P} := u_{X/Y, P}^{(m_{X/Y, P})}$ is the last ramification jump.
\end{itemize}
By Hasse--Arf theorem (cf. \cite[p. 76]{Serre1979}),
 if the $p$-Sylow subgroup of $G$ is abelian, the numbers $u_{X/Y, P}^{(t)}$ are integers.
For any $Q \in Y(\ol k)$ we denote also by abuse of notation $e_{X/Y, Q} := e_{X/Y, P}$,
$u_{X/Y, Q}^{(t)} := u_{X/Y, P}^{(t)}$, $G_Q := G_P$, etc. for arbitrary $P \in \pi^{-1}(Q)$.
Note that $G_Q$ is well-defined only up to conjugacy.
Let
\[
	B_{X/Y} := \{ Q \in Y(\ol k) : e_{X/Y, Q} > 1 \}
\]
be the ramification locus of $\pi$. 
We recall now the classical Chevalley-Weil formula. Let $\theta_{X/Y, P} : G_P \to \Aut_k(\mf m_P/\mf m_P^2) \cong k^{\times}$ be the fundamental character of~$P$.
In other words, if $t_P$ is a uniformizer at $P$ then for any $\sigma \in G_P$:
\[
	\theta_{X/Y, P}(\sigma) \equiv \frac{\sigma(t_P)}{t_P} \pmod{\mf m_P^2}.
\]
 Again, for $Q \in Y(k)$ we write $\theta_{X/Y, Q} := \theta_{X/Y, P}$ for any $P \in \pi^{-1}(Q)$. For any indecomposable $k[G]$-module
$W$ we denote by $N_{Q, i}(W)$ the multiplicity of the character $\theta_{X/Y, P}^i$ in the $k[G_Q]$-module $W|_{G_Q}$. In the sequel we often use the Iverson bracket notation:
\[
\llbracket \ms P \rrbracket = 
\begin{cases}
	1, & \textrm{ if the statement $\ms P$ is true,}\\
	0, & \textrm{ otherwise.}
\end{cases}
\]
\begin{Proposition} \label{prop:chevalley_weil}
	Keep the above notation and assume that $p \nmid \# G$.  Let also $D = \sum_{P \in X(k)} m_P \cdot P \in \Divv(X)$ be an effective $G$-equivariant divisor.
	Write $m_Q := m_P$ for any $P \in \pi^{-1}(Q)$.
	Then:
	\begin{equation} \label{eqn:cw}
		H^0(X, \Omega_X(D)) \cong \bigoplus_{W \in \Indec(k[G])} W^{\oplus a(W)},
	\end{equation}
	where:
	\begin{align*}
		a(W) &:=
		\left( (g_Y - 1)  + \frac{\deg D}{\# G}
		\right) \cdot \dim_k W
		\\ &
		+ \sum_{Q \in Y(k)} \sum_{i = 1}^{e_{X/Y, Q} - 1} \left\langle \frac{ -m_Q - i}{e_{X/Y, Q}} \right\rangle \cdot N_{Q, i}(W)+ \llbracket D = 0 \rrbracket \cdot \llbracket W \cong k \rrbracket.
	\end{align*}
\end{Proposition}
\begin{proof}
	For $D \neq 0$ this follows from \cite[Theorem~3.8]{Ellingsrud_Lonsted_Equivariant_Lefschetz} by noting that $H^1(X, \Omega_X(D)) = 0$
	and that the multiplicity of $\theta_{X/Y, Q}$ in 
\[
\Omega_X(D)_P/\mf m_P \Omega_X(D)_P = \Span_k(t_P^{-m_Q} \, dt_P)
\]
 equals $1 - m_Q$ for $P \in \pi^{-1}(Q)$. For $D = 0$ the same reasoning applies by noting that $H^1(X, \Omega_X) \cong k$
	is a trivial $k[G]$-module.
\end{proof}
Proposition~\ref{prop:chevalley_weil} and Serre's duality imply that under the above assumptions the $k[G]$-module structure of $H^1(X, \mc O_X(-D)) \cong H^0(X, \Omega_X(D))^{\vee}$ is uniquely determined by the ramification data. Note also that the de Rham cohomology of $X$ satisfies the Hodge--de Rham exact sequence:
\begin{equation} \label{eqn:hdr_exact_sequence}
	0 \to H^0(X, \Omega_X) \to H^1_{dR}(X) \to H^1(X, \mc O_X) \to 0.
\end{equation}
Thus, since the category of $k[G]$-modules is semisimple under the above assumptions, the ramification data determines the structure of $H^1_{dR}(X)$ as well.\\
\section{Modular representation theory} \label{sec:rep_theory}
In this section we review the modular representation theory in characteristic $p$ for a group~$G$ with a normal cyclic $p$-Sylow subgroup $H$.
By the Schur--Zassenhaus theorem (cf. \cite[Theorem~7.41]{Rotman_intro_to_groups}), $G$~is of the form:
\begin{equation} \label{eqn:form_of_group}
	G \cong H \rtimes_{\chi} C \quad  \textrm{with } H \cong \ZZ/p^n, \, p \nmid \# C \textrm{ and } \chi : C \to \Aut(H).
\end{equation}
 Let $H = \langle \sigma \rangle$. Observe that since $|\Aut(H)| = p^{n-1} \cdot (p-1)$,
$\chi$ must factor through $\Aut(\ZZ/p) \cong (\ZZ/p)^{\times}$.
Thus we can identify $\chi$ with a homomorphism $\chi : C \to \FF_p^{\times} \subset k^{\times}$.
Under this identification, we have $\rho \sigma \rho^{-1} = \sigma^{\chi(\rho)}$ for any $\rho \in C$. 
The representation theory of groups of the form~\eqref{eqn:form_of_group} is well-understood.
Assume that $k$ is algebraically closed of characteristic~$p$. Let $\Indec(k[G])$ be the set of isomorphism classes
of indecomposable $k[G]$-modules.
For every $k[G]$-module $U$, we will denote by $U_H$ the $k[H]$-module with restricted action. By \cite[Lemma~8, p.~34]{Alperin_local_rep} we have that $\mathrm{rad}(U)=\mathrm{rad}(U_H)=(1-\sigma)U$.  Hence the socle of~$U$ is $\mathrm{soc}(U)=\{u\in U: \mathrm{rad}(k[G])u=0\}$. Equivalently, it consists of the elements $u\in U$ such that $(1-\sigma)u=0$; that is $\mathrm{soc}(U)=U^{\sigma}$.

Every indecomposable module $U$ is uniserial by \cite[p. 42]{Alperin_local_rep} and is characterized by the composition length and the simple module $U/\mathrm{rad}(U)=S$. This means that there is a unique indecomposable projective module~$P$ that corresponds to $S$, i.e. $P/\mathrm{rad}(P)=S$, \cite[thm. 3, p. 31]{Alperin_local_rep}. Then by \cite[lemma 5, p.34]{Alperin_local_rep} we have that $U$ is a homomorphic image of $P$. Moreover $\mathrm{soc}(P)=\mathrm{soc}(U)=S$ by \cite[prop. 5, p.6]{Alperin_local_rep}. This proves that 
$S=U/\mathrm{rad}U=P/\mathrm{rad}(P)=\mathrm{soc}(P)=\mathrm{soc}(U)$, by \cite[thm. 6, p.43]{Alperin_local_rep}. 

Assume that $U^{\sigma}=S$ has dimension $\dim_k U^{\sigma}=d$. Then by \cite[p.~35]{Alperin_local_rep} we have that each successive quotient of $P$ and $U$ is $d$-dimensional. This proves that the decomposition  length equals $\dim_kU/\dim_kU^{\sigma}$. \\

We summarize the above discussion in the following well-known result.
\begin{Proposition}
For $U \in \Indec(k[G])$, the socle $U^{\sigma}= \ker(\sigma - 1)$  belongs to
$\Indec(k[C])$. The map
\begin{align*}
	\Indec(k[G]) &\to \Indec(k[C]) \times \{ 1, \ldots, p^n \}\\
	U &\mapsto \left(U^{\sigma}, \frac{\dim_k U}{\dim_k U^{\sigma}} \right)
\end{align*}
is a bijection. 
\end{Proposition}
We write $J_i(M)$ for the $k[G]$-module corresponding to a pair $(M, i) \in \Indec(k[C]) \times \{ 1, \ldots, p^n \}$.
Moreover, if $V$ is $k[C]$-module with the decomposition $V \cong \bigoplus_{M \in \Indec(k[C])} M^{\oplus a_M}$ then we put:
\[
J_i(V) := \bigoplus_{M \in \Indec(k[C])} J_i(M)^{\oplus a_M}.
\]
For any $k[H]$-module~$M$ denote:
\begin{align*}
	M^{(i)} &:= \ker ((\sigma - 1)^i : M \to M),\\
	T^i_H M &:= M^{(i)}/M^{(i-1)} \quad \textrm{ for } i = 1, \ldots, p^n.
\end{align*}
Note that if $M$ is a $k[G]$-module, then $T^i_H M$ is a $k[C]$-module for $i = 1, \ldots, p^n$.
The modules $T^1_H M, \ldots, T^{p^n}_H M$ were crucial in the proof of \cite[Theorem~1.1]{Bleher_Chinburg_Kontogeorgis_Galois_structure}.
Their importance lies in the fact that they determine~$M$ completely (see Lemma~\ref{lem:properties_TiM} (5) below).
We give now several facts concerning relations between those modules. To this end we need to introduce some notation.
For any $k[C]$-module $M$ and any homomorphism $\psi : C \to k^{\times}$ we write $M^{\psi} := M \otimes_{k[C]} \psi$.
In what follows, we fix $n$ and write $T^i M := T^i_H M$. 
 Let $H'$ (respectively $H''$) be the unique subgroup (respectively quotient) of $H$ isomorphic to $\ZZ/p^{n-1}$.
\begin{Lemma} \label{lem:properties_TiM}
Let $M$ be a $k[G]$-module of a finite dimension.
\begin{enumerate}[leftmargin=*]
	\item For any $1 \le i \le p^n - 1$ there exists a $k[C]$-equivariant monomorphism:
	\[
	m_{\sigma - 1} : T^{i+1} M \hookrightarrow (T^i M)^{\chi^{-1}}.
	\]
	\item For any $1 \le i \le p^n - 1$ there exists an isomorphism of $k[C]$-modules:
	\[
		T^i_{H'} \, M \cong T^{pi - p + 1} M \oplus T^{pi - p + 2} M \oplus \ldots \oplus T^{pi - p} M.
	\]

	\item If $\dim_k T^i_{ H'} M = \dim_k T^{i+1}_{ H'} M$  for some $1 \le i \le p^{n-1} - 1$ then 
	\[
		T^{j + 1} M \cong (T^j M)^{\chi^{-1}} \quad \textrm{ for } j = pi - p + 1, \ldots, pi + p - 1.
	\]
	\item For any $j = 1, 2, \ldots, p^{n-1}$ the trace
	\[
	\tr_{\langle \sigma^{p^{n-1}} \rangle} := \sum_{i = 0}^{p-1} (\sigma^{p^{n-1}})^i
	\]
	induces a $k[C]$-equivariant monomorphism:
	\[
	T^{p^n - p^{n-1} + j} M \hookrightarrow T^j_{H''} \, M^{(p^{n-1})}.
	\]
	\item The $k[G]$-structure of $M$ is uniquely determined by the $k[C]$-structure of $T^1 M, \ldots, T^{p^n} M$.
\end{enumerate}
\end{Lemma}
\begin{proof}
	(1) We define $m_{\sigma - 1} : T^{i+1} M \to T^i M$ as follows:
	\[
	m_{\sigma - 1}(\ol x) := (\sigma - 1) \cdot x,
	\]
	where for $\ol x \in T^{i+1} M$ we picked any representative $x \in M^{(i+1)}$.
	Indeed, if $x \in M^{(i+1)}$ then clearly $(\sigma - 1) \cdot x \in M^{(i)}$.
	Moreover $(\sigma - 1) \cdot x \in M^{(i-1)}$ holds if and only if $x \in M^{(i)}$. This
	shows that $m_{\sigma - 1}$ is well-defined and injective.
	Thus it suffices to check that it is $\chi^{-1}$-linear.
	Note that we have the following identity in the ring~$k[G]$ for any $\rho \in C$:
	\[
	(\sigma - 1) \cdot \rho = \rho \cdot (\sigma^{\chi(\rho)^{-1}} - 1)
	= \rho \cdot (\sigma - 1) \cdot (1 + \sigma + \sigma^2 + \cdots + \sigma^{\chi(\rho)^{-1} - 1})
	\]
	Note that $\sigma$ acts trivially on $T^i M$, so that for any $\ol x \in T^i M$:
	\[
	(1 + \sigma + \sigma^2 + \cdots + \sigma^{\chi(\rho)^{-1} - 1}) \cdot \ol x = \chi(\rho)^{-1} \cdot \ol x.
	\]
	This easily shows that
	\[
	m_{\sigma - 1}(\rho \cdot \ol x) = \chi(\rho)^{-1} \cdot \rho \cdot m_{\sigma - 1}(\ol x),
	\]
	which ends the proof.\\ \mbox{} \\
	(2) Note that $T_{H'}^i \, M = M^{(pi)}/M^{(pi - p)}$. The proof follows, since the category of $k[C]$-modules
	is semisimple and $T_{H'}^i \, M$ has a filtration $M^{(pi - p + j)}/M^{(pi - p)}$ ($j = 0, \ldots, p$) with $T^{pi - p + j} M$ as subquotients.\\ \mbox{} \\
	\noindent (3) By~(1) and~(2) we obtain:
	\begin{align*}
		\dim_k T_{H'}^i \, M &= \dim_k T^{pi} M + \cdots + \dim_k T^{pi - p + 1} M\\
		&\ge \dim_k T^{pi+p} M + \cdots + \dim_k T^{pi+1} M
		= \dim_k T_{H'}^{i+1} M.
	\end{align*}
	Since the left-hand side and right hand side are equal, we obtain:
	\[
	\dim_k T^{pi + p} M = \dim_k T^{pi + p - 1} M = \cdots = \dim_k T^{pi - p + 1} M.
	\]
	The proof follows by part~(1).\\ \mbox{} \\
	(4) Recall that in $\FF_p[x]$ we have the identity:
	\[
	1 + x + \cdots + x^{p-1} = (x - 1)^{p-1}.
	\]
	Therefore in the group ring $k[H]$ we have:
	\[
	\tr_{\langle \sigma^{p^{n-1}} \rangle} = \sum_{j = 0}^{p-1} (\sigma^{p^{n-1}})^j = (\sigma^{p^{n-1}} - 1)^{p-1} = 
	(\sigma - 1)^{p^n - p^{n-1}}.
	\]
	This implies that:
	\[
	\ker(\tr_{\langle \sigma^{p^{n-1}} \rangle} : M \to M^{(p^{n-1})}) = M^{(p^n - p^{n-1})}
	\]
	and that $\tr_{\langle \sigma^{p^{n-1}} \rangle}$ induces the desired monomorphism.
	Similarly as in part~(1) one shows that it is $\chi^{-(p^n - p^{n-1})}$-linear.
	Finally, note that $\chi^{p^n - p^{n-1}} {=} 1$.\\ \mbox{} \\
	(5) This is basically \cite[proof of Theorem~1.1]{Bleher_Chinburg_Kontogeorgis_Galois_structure}. We sketch the proof for reader's convenience. Write
	\[
	M \cong \bigoplus_{i = 1}^{p^n} \bigoplus_{W \in \Indec(C)}  J_i(W)^{\oplus n(W, i)}.
	\]
	Note that as $k[C]$-modules:
	\[
	T^j  J_i(W) \cong
	\begin{cases}
		W^{\chi^{-j + 1}}, & \textrm{ if } j \le i,\\
		0, & \textrm{ if } j > i.
	\end{cases}
	\]
	Hence:
	\[
	T^j M \cong \bigoplus_{i = j}^{p^n} \bigoplus_{W \in \Indec(C)} (W^{\chi^{-j + 1}})^{\oplus n(W, i)}
	\]
	and the $k[C]$-module structure of $T^j M$ determines uniquely 
	the numbers:
	\[
	\sum_{i = j}^{p^n} n(W, i)
	\]
	for every $W \in \Indec(k[C])$. This easily implies that the numbers $n(W, 1)$, $\ldots$, $n(W, p^n)$ are uniquely determined by the $k[C]$-structure of $T^1 M$, $\ldots$, $T^{p^n} M$. The proof follows.
\end{proof}
If $G = H \cong \ZZ/p^n$ we abbreviate $J_i := J_i(k)$. Observe that $J_i$ is given by the Jordan block of size $i$ and eigenvalue $1$.
Moreover, for $i > 0$:
\begin{equation} \label{eqn:dim_of_Ti_Jl}
	\dim_k T^i J_l = \llbracket i \le l \rrbracket.
\end{equation}
Lemma~\ref{lem:properties_TiM}~(5) implies that the numbers $\dim_k T^i M$ for $i=1, \ldots, p^n$ determine the structure of a $k[H]$-module $M$ completely.
If $H'$ is as above, we denote the indecomposable $k[H']$-modules by $\mc J_1, \ldots, \mc J_{p^{n-1}}$.

\section{Cyclic covers} \label{sec:cyclic}
The goal of this section is to prove the following formula for the de Rham cohomology
of $\ZZ/p^n$-covers of curves.
\begin{Theorem} \label{thm:cyclic_de_rham}
	Let $k$ be an algebraically closed field of characteristic~$p$.
	Suppose that $\pi : X \to Y$ is a $\ZZ/p^n$-cover. Pick an arbitrary $Q_0 \in Y(k)$
	with $m_{X/Y, Q_0} = m_{X/Y}$. Then, as a $k[\ZZ/p^n]$-module
	$H^1_{dR}(X)$ is isomorphic to:
	\begin{equation} \label{eqn:HdR_formula}
		J_{p^n}^{2 (g_Y - 1)} \oplus J_{p^n - p^{n-m} + 1}^2 \oplus \bigoplus_{\substack{Q \in B\\ Q \neq Q_0}} J_{p^n - p^n/e_{Q}}^2
		\oplus \bigoplus_{Q \in B} \bigoplus_{t = 0}^{m_Q  - 1} J_{p^n - p^{n+t}/e_Q}^{u_Q^{(t+1)} - u_Q^{(t)}},
	\end{equation}
	where we abbreviate $B := B_{X/Y}$, $e_Q := e_{X/Y, Q}$, $u_Q^{(t)} := u_{X/Y, Q}^{(t)}$, $m := m_{X/Y, Q}$ and $m_Q := m_{X/Y, Q}$.
\end{Theorem}
\begin{Remark}
	Note that this formula is well-defined for $g_Y = 0$, even though the first exponent is negative. Indeed, since $m_{X/Y} = n$ (as $\PP^1$ doesn't have any \'{e}tale covers), the first two summands in~\eqref{eqn:HdR_formula} cancel out.
\end{Remark}
\noindent We prove now several auxiliary facts used in the proof of Theorem~\ref{thm:cyclic_de_rham}.
\begin{Lemma} \label{lem:G_invariants_\'{e}tale}
	If $G$ is a finite group and the $G$-cover $\pi : X \to Y$ is \'{e}tale, then
	\[
		\dim_k H^1_{dR}(X)^G = 2g_Y - \dim_k H^1(G, k) + \dim_k H^2(G, k).
	\]
	In particular, if $G \cong \ZZ/p^n$ then $\dim_k H^1_{dR}(X)^G = 2g_Y$.
\end{Lemma}
\begin{proof}
	Let $\HH^i(Y, \mc F^{\bullet})$ be the $i$th hypercohomology of a complex $\mc F^{\bullet}$.  
	Write also $\mc H^i(G, -)$ for the $i$th derived functor of the functor
	\[
		\mc F \mapsto \mc F^G.
	\]
	Since $X \to Y$ is \'{e}tale, $\mc H^i(G, \pi_* \mc F) = 0$ for any $i > 0$ and any coherent sheaf $\mc F$ on $X$ by \cite[Proposition~2.1]{Garnek_equivariant}.
	Therefore the spectral sequence~\cite[(3.4)]{Garnek_equivariant} applied for the complex $\mc F^{\bullet} := \pi_* \Omega_{X/k}^{\bullet}$ yields $\RR^i \Gamma^G(\pi_* \Omega_{X/k}^{\bullet}) = \HH^1(Y, \pi_*^G \Omega_{X/k}^{\bullet}) = H^1_{dR}(Y)$, since $\pi_*^G \Omega_X^{\bullet} \cong \Omega_Y^{\bullet}$ (cf. \cite[Corollary~2.4]{Garnek_equivariant}).
	On the other hand, the seven-term exact sequence applied for the spectral sequence~\cite[(3.5)]{Garnek_equivariant} yields:
	\begin{align*}
		0 \to H^1(G, H^0_{dR}(X)^G) \to H^1_{dR}(Y) \to H^1_{dR}(X)^G \to H^2(G, H^0_{dR}(X)^G) \to K,
	\end{align*}
	where:
	\[
		K := \ker(H^2_{dR}(Y) \to H^2_{dR}(X)^G) = \ker(k \stackrel{\id}{\rightarrow} k) = 0.
	\]
	Therefore, since $H^0_{dR}(X)^G \cong k$:
	\begin{align*}
		\dim_k H^1_{dR}(X)^G = \dim_k H^1_{dR}(Y) - \dim_k H^1(G, k) + \dim_k H^2(G, k)\\
		= 2g_Y - \dim_k H^1(G, k) + \dim_k H^2(G, k).
	\end{align*}
	Finally, note that if $G$ is cyclic then $\dim_k H^1(G, k) = \dim_k H^2(G, k)$ by \cite[th. 6.2.2]{Weibel}.
\end{proof}
\begin{Remark}
The equality $\dim_k H^1(G, k) = \dim_k H^2(G, k)$ does not hold for non-cyclic groups. For example it is known \cite[cor. II.4.3,th. II.4.4]{MR2035696} that the cohomological ring for the elementary abelian group $G = (\ZZ/p)^s$ is given by 
\[
	H^* (G, \mathbb{F}_p)=
	\begin{cases}
		\mathbb{F}_2[x_1, \ldots,x_s] & \text{ if } p=2 \\
		\wedge(x_{1}, \ldots, x_s) \otimes \mathbb{F}_p[x_1, \ldots,x_s] & \text{ if } p>2
	\end{cases}
\]
Therefore, for $s>1$ the degree one and two parts of the cohomological ring, which correspond to the first and second cohomology groups, have different dimensions. 
\end{Remark}
\begin{Lemma} \label{lem:trace_surjective}
	Suppose that $G$ is a $p$-group. If the $G$-cover $X \to Y$ has no \'{e}tale subcovers, then the map
	\[
	\tr_{X/Y} : H^1_{dR}(X) \to H^1_{dR}(Y)
	\]
	is an epimorphism.
\end{Lemma}
\begin{proof}
By induction, it suffices to prove this in the case when $G = \ZZ/p$.
Consider the following commutative diagram:
\begin{equation} \label{eqn:diagram_trace}
	\begin{tikzcd}
		0 \arrow[r] & {H^0(X, \Omega_X)} \arrow[r] \arrow[d, "\tr_{X/Y}"] & H^1_{dR}(X) \arrow[r] \arrow[d, "\tr_{X/Y}"] & {H^1(X, \mc O_X)} \arrow[r] \arrow[d, "\tr_{X/Y}"] & 0 \\
		0 \arrow[r] & {H^0(Y, \Omega_Y)} \arrow[r]                        & H^1_{dR}(Y) \arrow[r]                        & {H^1(Y, \mc O_Y)} \arrow[r]                        & 0
	\end{tikzcd}
\end{equation}
where the rows are Hodge--de Rham exact sequences (cf.~\eqref{eqn:hdr_exact_sequence}). Recall that by~\cite[Theorem~1]{Valentini_Madan_Automorphisms}, in this case $H^0(X, \Omega_X)$ contains
a copy of $k[G]^{\oplus g_Y}$ as a direct summand. Thus, since trace is injective on $k[G]^{\oplus g_Y}$, the dimension
of the image of
\begin{equation} \label{eqn:trace_H0_Omega}
	\tr_{X/Y} : H^0(X, \Omega_X) \to H^0(Y, \Omega_Y)
\end{equation}
is $g_Y$. Therefore the map~\eqref{eqn:trace_H0_Omega} is surjective. 
Similarly, by Serre's duality, $H^1(X, \mc O_X)$ also contains $k[G]^{\oplus g_Y}$ as a direct summand
and hence the trace map
\begin{equation*} 
	\tr_{X/Y} : H^1(X, \mc O_X) \to H^1(Y, \mc O_Y)
\end{equation*}
is surjective. Therefore, since the outer vertical maps in the diagram~\eqref{eqn:diagram_trace} are surjective,
the trace map on the de Rham cohomology must be surjective as well.
\end{proof}

\begin{Lemma} \label{lem:u_equals_ul}
	Assume that $\phi: Y' \to Y$ is a $\ZZ/p$-subcover of $X \to Y$.
	Then:
	\[
	p \cdot \sum_{Q \in B_{X/Y}} (u_{X/Y, Q} - 1) = \sum_{Q' \in B_{X/Y'}} (u_{X/Y', Q'} - 1)
	+ (p-1) \cdot \sum_{Q \in B_{Y'/Y}} (l^{(1)}_{Y'/Y, Q}  - 1).
	\]
\end{Lemma}
\begin{proof}
	Pick a point $Q \in B_{X/Y}$. If $Q \not \in B_{Y'/Y}$ then 
	$u_{X/Y, Q} = u_{X/Y', Q'}$ for all $p$ points $Q' \in Y'(k)$ in the preimage of $Q$ and:
	\begin{equation} \label{eqn:Q_not_in_B'}
		p \cdot (u_{X/Y, Q} - 1) = \sum_{Q' \in \phi^{-1}(Q)} (u_{X/Y', Q'} - 1).
	\end{equation}
	Assume now that $Q \in B_{Y'/Y}$. Then there exists a unique point $Q' \in Y'(k)$
	in the preimage of $Q$ through $\phi: Y' \to Y$. Moreover, $m_{X/Y, Q} = n$, $m_{X/Y', Q'} = n-1$. 
	Recall also that by \cite[Example p.76]{Serre1979}
	there exist integers $i_{X/Y, P}^{(0)}, i_{X/Y, P}^{(1)}, \ldots$ such that for every $t \ge 0$:
	\begin{align*}
		u_{X/Y, P}^{(t)} &= i_{X/Y, P}^{(0)} + i_{X/Y, P}^{(1)} + \cdots + i_{X/Y, P}^{(t-1)}\\
		l_{X/Y, P}^{(t)} &= i_{X/Y, P}^{(0)} + i_{X/Y, P}^{(1)} \cdot p + \cdots + i_{X/Y, P}^{(t-1)} \cdot p^{t-1}.
	\end{align*}
	Observe that:
	\begin{align*}
		i_{X/Y' , P}^{(0)} &= i_{X/Y, P}^{(0)} + i_{X/Y, P}^{(1)} \cdot p,\\
		i_{X/Y' , P}^{(t)} &= p \cdot i_{X/Y, P}^{(t + 1)}  \quad \textrm{ for } t > 0.
	\end{align*}
	This implies that
	\begin{equation} \label{eqn:Q_in_B'}
		p \cdot (u_{X/Y, Q} - 1) = (u_{X/Y', Q'} - 1) + (p-1) \cdot (l^{(1)}_{X/Y, Q} - 1).
	\end{equation}
	Indeed, using the above formulas:
	\begin{align*}
		p \cdot (u_{X/Y, Q} - 1) &=
		p \cdot (i^{(0)}_{X/Y, Q} + \cdots + i^{(m_Q - 1)}_{X/Y, Q} - 1)\\
		&= (p-1) \cdot (i^{(0)}_{X/Y, Q} - 1) + (i^{(0)}_{X/Y, Q} + p \cdot i^{(1)}_{X/Y, Q}) \\
		&+ p \cdot (i^{(2)}_{X/Y, Q} + i^{(3)}_{X/Y, Q} + \cdots +i^{(m_Q - 1)}_{X/Y, Q}) - 1 \\
		&= (p-1) \cdot (l^{(1)}_{X/Y, Q} - 1) + (i^{(0)}_{X/Y', Q'} + i^{(1)}_{X/Y', Q'} + \cdots - 1)\\
		&= (p-1) \cdot (l^{(1)}_{X/Y, Q}  - 1) + (u_{X/Y', Q'} - 1).
	\end{align*}
	The proof follows by summing~\eqref{eqn:Q_not_in_B'} and~\eqref{eqn:Q_in_B'} over all $Q \in B_{X/Y}$.
\end{proof}

\begin{proof}[Proof of Theorem~\ref{thm:cyclic_de_rham}]
	We proceed by induction on $n$. We define $H'$, $H''$, $T^i M$, $J_i$ and $\mc J_i$ as in Section~\ref{sec:rep_theory}. Also, we write $Y' := X/H'$ and $X'' := X/\langle \sigma^{p^{n-1}} \rangle$, see the diagram below. 
	\[
	\xymatrix{
		& X \ar[rd]^{\langle \sigma^{p^{n-1}} \rangle} \ar[ld]_{\langle \sigma^p \rangle =H'} \ar[dd]^{\pi}&  \\
		Y'  \ar[rd]^{\phi} &   &  X'' \ar[ld] ^{H''}\\ 
		& Y &
	}
	\]
	Note that $H''$ naturally acts on $X''$ and $X''/H'' \cong Y$.
	Let also $\mc M := H^1_{dR}(X)$ and $\mc M'' := H^1_{dR}(X'')$.
	Write $\mc M_0$ for the module~\eqref{eqn:HdR_formula}.
	 Note that the trace map $\tr_{\langle \sigma^{p^{n-1}} \rangle} : \mc M \to \mc M$
	equals the composition of $\tr_{X/X''} : H^1_{dR}(X) \to H^1_{dR}(X'')$
	with the inclusion of $H^1_{dR}(X'') \hookrightarrow H^1_{dR}(X)$. Moreover, 	
	by Lemma~\ref{lem:trace_surjective} if $X \to X''$
	isn't \'{e}tale then $\tr_{X/X''} : H^1_{dR}(X) \to H^1_{dR}(X'')$ is surjective.
	Hence, by Lemma~\ref{lem:properties_TiM}~(4) $\tr_{X/X'}$ yields a $k$-linear isomorphism:
	\begin{equation} \label{eqn:trace_iso}
		T^{j + p^n - p^{n-1}} \mc M \cong T^j_{H''} \, \mc M''.
	\end{equation}
	The proof of Theorem~\ref{thm:cyclic_de_rham} is divided into three cases.\\ \mbox{}\\
	\bb{Case I: the cover $X \to Y$ is \'{e}tale.}\\
	By Lemma~\ref{lem:G_invariants_\'{e}tale} and \cite[Corollary~2.4]{Garnek_equivariant} we have $\dim_k H^1_{dR}(X)^H = 2g_Y = \dim_k H^0(X, \Omega_X)^H + \dim_k H^1(X, \mc O_X)^H$. Therefore the Hodge--de Rham exact sequence splits by \cite[Lemma~5.3]{Garnek_equivariant} and
	\begin{align*}
		H^1_{dR}(X) &\cong H^0(X, \Omega_X) \oplus H^1(X, \mc O_X)\\
		&\cong H^0(X, \Omega_X) \oplus H^0(X, \Omega_X)^{\vee}
	\end{align*}
	(the last isomorphism follows from Serre's duality~\cite[Corollary III.7.7]{Hartshorne1977}).
	 Now it suffices to note that $H^0(X, \Omega_X) \cong J_{p_n}^{g_Y - 1} \oplus k$ by \cite{Tamagawa:51} (see also \cite{Kani:86} and \cite[Theorem~1]{Valentini_Madan_Automorphisms}).\\ \mbox{}\\
	\bb{Case II: $n = 1$, the cover $X \to Y$ is not \'{e}tale.}\\
	Note that $\mc M_0 = J_p^{\oplus 2g_Y} \oplus J_{p-1}^{\oplus \alpha}$,
	where $\alpha := \sum_{Q \in B_{X/Y}} (u_Q + 1) - 2$. Write
	\[
	\mc M \cong \bigoplus_{i = 1}^p J_i^{\oplus a_i}.
	\]
	Observe that 
	\begin{equation} \label{eqn:dr_invariants_for_Zp}
	\sum_{i = 1}^p a_i = \dim_k \mc M^H = 2 g_Y + \alpha.
	\end{equation}
	%
	Indeed, \cite[Theorem 1.2]{Garnek_equivariant} yields the following formula:
	\begin{align*}
		\delta &:=	\dim_k H^1_{dR}(X)^H - \dim_k H^0(X, \Omega_X)^H - \dim_k H^1(X, \mc O_X)^H\\
		&= \sum_{Q \in Y(k)} \left( u_Q - 1-2 \left[ \frac{u_Q}{p}\right]
		\right).
	\end{align*}
	Put:
	\begin{equation} \label{eqn:R'_definition}
		R' := \sum_{Q \in B_{X/Y}} \left\lfloor \frac{(u_Q + 1) \cdot (p-1)}{p} \right\rfloor \cdot Q \in \Divv(Y).
	\end{equation}
	Then, by \cite[Corollary 3.5]{Garnek_equivariant} and by Riemann--Roch theorem we have 
	\begin{align*}
		\dim_k H^0(X, \Omega_X)^H = \dim_k H^0(Y, \Omega_Y(R')) = g_Y - 1 + \deg R'.
	\end{align*}
	Moreover, for any $k[H]$-module we have $\dim_k V^H = \dim_k V_H$ (one checks this easily by analyzing $V$ as a direct sum of Jordan blocks) and $(V^H)^{\vee} \cong (V^{\vee})_H$. Hence:
	\begin{align*}
		\dim_k H^1(X, \mc O_X)^H &= \dim_k \left( H^1(X, \mc O_X)^H \right)^{\vee}\\
		&= \dim_k H^0(X, \Omega_X)_H = \dim_k H^0(X, \Omega_X)^H.
	\end{align*}
	Finally:
	\begin{align*}
		\dim_k H^1_{dR}(X)^H &= \delta + \dim_k H^0(X, \Omega_X)^H + \dim_k H^1(X, \mc O_X)^H\\
		&= \sum_{Q \in Y(k)} \left( u_Q - 1-2 \left[ \frac{u_Q}{p}\right]
		\right) + 2 \cdot (g_Y - 1 + \deg R')\\
		&= 2g_Y + \alpha,
	\end{align*}
	which proves~\eqref{eqn:dr_invariants_for_Zp}.
	Moreover,  by~\eqref{eqn:trace_iso} $T^p \mc M \cong \mc M'' \cong H^1_{dR}(Y)$, i.e. $a_p = 2 g_Y$. 
	

	%
	Finally:
	\[
		\sum_{i = 1}^p i \cdot a_i = \dim_k \mc M = 2 g_Y \cdot p + \alpha \cdot (p-1).
	\]
	Hence:
	\begin{align*}
		\sum_{i = 1}^{p - 1} (p-1 - i) \cdot a_i &= (p-1) \sum_{i = 1}^p a_i - \sum_{i = 1}^p i \cdot a_i + a_p\\
		&= (p-1) \cdot (2 g_Y + \alpha) - (2 g_Y \cdot p + \alpha \cdot (p-1)) + 2 g_Y\\
		&= 0.
	\end{align*}
	However, $p-1 - i \ge 0$ with equality only for $i = p-1$,
	which implies that $a_1 = \cdots = a_{p-2} = 0$ and $a_{p-1} = \alpha$. This ends the proof in this case. \\ \mbox{}\\
	\bb{Case III: $n > 1$, the cover $X \to Y$ is not \'{e}tale.}\\
	The assumption implies that $X \to X''$ is also not \'{e}tale, since if $\sigma$ fixes a point then this point is also fixed by $\sigma^{p^n-1}$.
	We now show that $\dim_k T^i \mc M = \dim_k T^i \mc M_0$ separately for the cases $i \le p^n - p^{n-1}$ and $i > p^n - p^{n-1}$.
	By induction hypothesis for $H'$ acting on $X$, we have the following isomorphism of $k[H']$-modules:
	\begin{align}
		\mc M_{ H'} &\cong \mc J_{p^{n-1}}^{2 (g_{Y'} - 1)} \oplus \mc J_{p^{n-1} - 
			p^{n - 1 -  m'} + 1}^2 \oplus 
			\bigoplus_{\substack{Q \in Y'(k)\\Q \neq Q_1}} \mc J_{p^{n-1} - p^{n-1}/e'_Q}^2 \nonumber\\
		&\oplus \bigoplus_{Q \in Y'(k)} \bigoplus_{t = 0}^{ m_Q' - 1} \mc J^{u_{X/Y', Q}^{(t+1)} - u_{X/Y', Q}^{(t)}}_{ p^{n - 1} - p^{n - 1 + t}/e'_Q} \label{eqn:MH'_induction}
	\end{align}
	where $e'_Q := e_{X/Y', Q}$, $m'_Q := m_{X/Y', Q}$,  $m' := m_{X/Y'}$ and $Q_1 \in \pi^{-1}(Q_0)$. Note also that
	\begin{align}
		 \# B_{X/Y'} &= p \cdot \# \{ Q \in B_{X/Y}: e_{Y'/Y, Q} = 1 \} + \# \{ Q \in B_{X/Y}: e_{Y'/Y, Q} = p \} \nonumber \\
		& = p \cdot (\# B - \# B_{Y'/Y}) + \# B_{Y'/Y} \nonumber \\
		& = p \cdot \# B - (p-1) \cdot \# B_{Y'/Y}. \label{eqn:BXY'=pB-(p-1)B}
	\end{align}
	Therefore, for $i \le p^{n-1} - p^{n-2}$, using  induction hypothesis (cf.~\eqref{eqn:MH'_induction}):
	%
	\begin{align*}
		\dim_k T^i_{H'} \mc M &= 2(g_{Y'} - 1) + 2 + 2(\# B_{ X/Y'} - 1)
		+ \sum_{Q' \in Y'(k)} (u_{X/Y', Q'} - 1) \\
		\intertext{ (by the Riemann--Hurwitz formula, cf. \cite[Corollary~IV.2.4]{Hartshorne1977})}
		&= 2 p (g_Y - 1) + \sum_{Q' \in Y'(k)} (p-1) \cdot (l_{Y'/Y, Q'}^{(1)} + 1)\\
		&+ 2 + 2(\# B_{X/Y'} - 1) + \sum_{Q' \in Y'(k)} (u_{X/Y', Q'} - 1) \\
		&  = 2 p (g_Y - 1) + \sum_{Q' \in Y'(k)} (p-1) \cdot (l_{Y'/Y, Q'}^{(1)} - 1)  \\
		& + 2 \cdot (p-1) \cdot \# B_{Y'/Y} + 2 + 2(\# B_{X/Y'} - 1) + \sum_{Q' \in Y'(k)} (u_{X/Y', Q'} - 1) \\
		\intertext{ (by~\eqref{eqn:BXY'=pB-(p-1)B})}
		& = 2 p (g_Y - 1) + \sum_{Q' \in Y'(k)} (p-1) \cdot (l_{Y'/Y, Q'}^{(1)} - 1) + 2 \cdot p \cdot \# B  \\
		&  + \sum_{Q' \in Y'(k)} (u_{X/Y', Q'} - 1)  \\
		\intertext{ (using Lemma~\ref{lem:u_equals_ul})}\\
		&= p \cdot \left( 2(g_Y - 1) + 2\# B  + \sum_{Q \in Y(k)} (u_{X/Y, Q} - 1) \right).
	\end{align*}
	In particular, $\dim_k T^1_{H'} \mc M = \cdots = \dim_k T^{p^{n-1} - p^{n-2}}_{H'} \mc M$,
	which implies by Lemma~\ref{lem:properties_TiM}~(3) that 
	\[
		\dim_k T^1 \mc M = \ldots = \dim_k T^{p^n - p^{n-1}} \mc M.
	\]
	Noting that (cf. Lemma~\ref{lem:properties_TiM}~(2)):
	\[
		\dim_k \mc T^1_{H'} \mc M = \dim_k T^1 \mc M + \ldots + \dim_k T^p \mc M,
	\]
	we obtain that for any $1 \le i \le p^n - p^{n-1}$:
	\begin{align*}
		\dim_k T^i \mc M &= \frac{1}{p} \dim_k T^1_{H'} \, \mc M\\
		&= 2(g_Y - 1) + 2 + 2(\# B - 1) + \sum_{Q \in Y(k)} (u_{X/Y, Q} - 1)\\
		&= \dim_k T^i \mc M_0.
	\end{align*}
	Now we treat the case when $i > p^n - p^{n-1}$. Let $i = p^n - p^{n-1} + j$
	and let~$N$  be chosen so that $j \in (p^{n-1} - p^{N+1}, p^{n-1} - p^N]$. Then:
	\begin{align*}
		\dim_k T^j \mc M_0 &= 2 \cdot (g_Y - 1) + 2 \cdot \llbracket N \ge n - m \rrbracket\\
		&+ 2 \cdot \# \{ Q \in Y(k) \setminus \{Q_0\} : N \ge n - m_Q \}\\
		&+ \sum_{Q \in Y(k)} \sum_{t = 0}^{m_Q - 1} \llbracket N \ge n+t - m_Q \rrbracket \cdot (u_{Q}^{(t+1)} - u_{Q}^{(t)}).
	\end{align*}
	On the other hand, by~\eqref{eqn:trace_iso} and by induction assumption  for $X'' \to Y$:
	\begin{align*}
		\dim_k T^i \mc M &= \dim_k T^j_{H''} \, \mc M'' = 2 \cdot (g_Y - 1) + 2 \cdot \llbracket N \ge (n - 1) - (m - 1) \rrbracket\\
		&+ 2 \cdot \# \{ Q \in Y(k) \setminus \{Q_0\} : N \ge (n-1) - m_{X''/Y, Q} \}\\
		&+ \sum_{Q \in Y(k)} \sum_{t = 0}^{m_{X''/Y, Q}} \llbracket N \ge (n-1) + t - m_{X''/Y, Q} \rrbracket \cdot (u_{X''/Y, Q}^{(t+1)} - u_{X''/Y, Q}^{(t)})\\
		&= 2 \cdot (g_Y - 1) + 2 \cdot \llbracket N \ge n - m \rrbracket\\
		&+ 2 \cdot \# \{ Q \in Y(k) \setminus \{Q_0\} : N \ge n - m_Q \}\\
		&+ \sum_{Q \in Y(k)} \sum_{t = 0}^{m_Q - 1} \llbracket N \ge n+t - m_Q \rrbracket \cdot (u_{Q}^{(t+1)} - u_{Q}^{(t)})\\
		&= \dim_k T^i \mc M_0.
	\end{align*}
	This ends the proof.
\end{proof}
\section{Proof of Main Theorem} \label{sec:main_thm}
In this section we prove Main Theorem, assuming a formula
for the cohomology of a cover with a $p$-Sylow subgroup of order $p$.
We prove this formula in the next section (cf. Theorem~\ref{thm:Zp_formula}).
The following result allows us to reduce the problem to the case when 
$G$ is of the form~\eqref{eqn:form_of_group} with $C = \ZZ/c$, $p \nmid c$ and $k$ is an algebraically closed field.
\begin{Lemma} \label{lem:reductions}
	Let $G$ and $k$ be as in Main Theorem.
	Suppose $M$ is a finitely generated $k[G]$-module. 	
	\begin{enumerate}[leftmargin=*]
		\item The $k[G]$-module structure of $M$
		is uniquely determined by the restrictions $M|_H$ as $H$ ranges over all subgroups of $G$
		of the form~\eqref{eqn:form_of_group} with $C \cong \ZZ/c$, $p \nmid c$.
		
		\item The $k[G]$-module structure of $M$ is uniquely determined by the $\ol k[G]$-module structure of $M \otimes_k \ol k$.
	\end{enumerate}	
\end{Lemma}
\begin{proof}
	\begin{enumerate}[leftmargin=*]
		\item This follows easily from Conlon induction theorem (cf. \cite[Theorem~(80.51)]{Curtis_Reiner_Methods_II}), see e.g. \cite[Lemma~3.2]{Bleher_Chinburg_Kontogeorgis_Galois_structure}.
		
		\item This is \cite[Proposition~3.5. (iii)]{Bleher_Chinburg_Kontogeorgis_Galois_structure}. \qedhere
	\end{enumerate}
\end{proof}
\noindent The following simple lemma will be used in the proof of Main Theorem.
\begin{Lemma} \label{lem:N+Nchi+...}
	Keep the above notation. Let $M$, $N$ be $k[C]$-modules. 
	If
	\[
		N \cong M \oplus M^{\chi} \oplus \cdots \oplus M^{\chi^{p-1}},
	\]
	then  $M$ is uniquely determined by $N$.
\end{Lemma}
\begin{proof}
	Note that
	\[
	N \cong M^{\oplus 2} \oplus M^{\chi} \oplus M^{\chi^2} \oplus \cdots \oplus M^{\chi^{p-2}}.
	\]
	By tensoring this isomorphism by $\chi^i$ we obtain:
	\begin{align*}
		N^{\chi^i} \cong (M^{\chi^i})^{\oplus 2} \oplus M^{\chi^{i+1}} \oplus M^{\chi^{i+2}} \oplus \cdots \oplus M^{\chi^{i + p-2}}
		\cong (M^{\chi^i})^{\oplus 2} \oplus \bigoplus_{\substack{j = 0\\j \neq i}}^{p-2} M^{\chi^j}
	\end{align*}
	for $i = 0, \ldots, p-2$. Therefore:
	\begin{equation} \label{eqn:M+N=N}
		M^{\oplus p} \oplus N^{\chi} \oplus N^{\chi^2} \oplus \cdots \oplus N^{\chi^{p-2}} 
		\cong N^{\oplus (p-1)}.
	\end{equation}
	Indeed, for the proof of~\eqref{eqn:M+N=N} note that
	\begin{align*}
		M^{\oplus p} &\oplus N^{\chi} \oplus N^{\chi^2} \oplus \cdots \oplus N^{\chi^{p-2}}
		\cong M^{\oplus p} \oplus \bigoplus_{i = 1}^{p-2} \left((M^{\chi^i})^{\oplus 2}
		\oplus \bigoplus_{\substack{j = 0\\j \neq i}}^{p-2} M^{\chi^j} \right)\\
		&\cong \left( M^{\oplus 2} \oplus M^{\chi} \oplus M^{\chi^2} \oplus \cdots \oplus M^{\chi^{p-2}} \right)^{\oplus (p-1)}
		\cong N^{\oplus (p-1)}.
	\end{align*}
	The isomorphism~\eqref{eqn:M+N=N} clearly proves the thesis.
\end{proof}
\begin{proof}[Proof of Main Theorem]
	As explained at the beginning of this section, it suffices to show this in the case when
	$G$ is of the form~\eqref{eqn:form_of_group} and $k = \ol k$ by Lemma~\ref{lem:reductions}. Let $Y := X/H$. Similarly as in the proof of Theorem~\ref{thm:cyclic_de_rham}, we write $H' := \langle \sigma^p \rangle \cong \ZZ/p^{n-1}$,
	$H'' := H/\langle \sigma^{p^{n-1}} \rangle \cong \ZZ/p^{n-1}$, $Y' := X/H'$ and $X'' := X/\langle \sigma^{p^{n-1}} \rangle$. Observe that the ramification data of the covers $X'' \to Y$ and $X \to Y'$ depends only on the ramification data of $X \to Y$.\\
	
	We prove the result by induction on~$n$. For $n = 0$ this follows by Chevalley--Weil theorem.
	The rest of the proof is again divided into three cases.\\ \mbox{} \\
	\bb{Case I: the cover $X \to Y$ is \'{e}tale.}\\
	In this case we deduce as in the proof of Theorem~\ref{thm:cyclic_de_rham}
	that
	\begin{align*}
		H^1_{dR}(X) &\cong H^0(X, \Omega_X) \oplus H^0(X, \Omega_X)^{\vee}
	\end{align*}
	as $k[G]$-modules and hence the result follows by \cite[Theorem~1.1]{Bleher_Chinburg_Kontogeorgis_Galois_structure}.\\ \mbox{}\\
	\bb{Case II: $n = 1$, the cover $X \to Y$ is not \'{e}tale.}\\
	This case will be addressed in the next section (see Theorem~\ref{thm:Zp_formula}).\\ \mbox{}\\
	\bb{Case III: $n > 1$, the cover $X \to Y$ is not \'{e}tale.}\\
	Recall that by Lemma~\ref{lem:properties_TiM}~(5) the isomorphism class of $\mc M$
	is uniquely determined by the $k[C]$-modules $T^1 M, \ldots, T^{p^n} M$.
	Lemma~\ref{lem:properties_TiM}~(3) and Theorem~\ref{thm:cyclic_de_rham}
	yield an isomorphism of $k[C]$-modules:
	\begin{equation} \label{eqn:TiM=T1M_chi}
		T^{i+1} \mc M \cong (T^1 \mc M)^{\chi^{-i}}
	\end{equation}
	for $i < p^n - p^{n-1}$.  By Lemma~ \ref{lem:properties_TiM}~(2), for $i \le p^{ n-1} - p^{ n-2}$:
	\begin{align*}
		T^i_{H'} \mc M &\cong T^{pi - p + 1} \mc M \oplus \cdots \oplus T^{pi} \mc M\\
		&\cong T^1 \mc M \oplus (T^1 \mc M)^{\chi^{-1}} \oplus \cdots \oplus
		(T^1 \mc M)^{\chi^{-p}}.
	\end{align*}
	By induction assumption, the $k[C]$-module structure of $T^i_{H'} \mc M$ is uniquely determined by the ramification data. Thus, by Lemma~\ref{lem:N+Nchi+...} for $N := T^1 \mc M$ and by~\eqref{eqn:TiM=T1M_chi} the $k[C]$-structure of the modules $T^i \mc M$ is uniquely determined by the ramification data for $i \le p^n - p^{n-1}$.
	By Lemma~\ref{lem:properties_TiM}~(4) and~\eqref{eqn:trace_iso}, $\tr_{X/X'}$ yields an isomorphism:
	\begin{equation} \label{eqn:trace_isomorphism}
		T^{i + p^n - p^{n-1}} \mc M \cong T^i_{H''} \, \mc M''.
	\end{equation}
	Thus, by induction hypothesis for $\mc M''$, the $k[C]$-structure of $T^{i + p^n - p^{n-1}} \mc M$
	is determined by the ramification data as well.
\end{proof}
\section{Covers with Sylow subgroup of order $p$} \label{sec:Zp-sylow}
In this section we assume that $G$ is of the form~\eqref{eqn:form_of_group} with~$n = 1$.
Assume that $X$ is a smooth projective curve with an action of~$G$ and let $\pi : X \to Y := X/H$. The goal of this section is to prove the following variant of Chevalley--Weil formula in this context.
\begin{Theorem} \label{thm:Zp_formula}
	Keep the above notation. Assume that $k$ is algebraically closed. If $G$ acts on a curve $X$ and the cover $X \to Y$ is not \'{e}tale, then:
	\[
	H^1_{dR}(X) \cong J_p(V_1) \oplus J_{p-1}(V_2),
	\]	
	where the $k[C]$-modules $V_1$ and $V_2$ are determined by the isomorphisms:
	\begin{align*}
		V_1 &\cong H^0(Y, \Omega_Y) \oplus H^1(Y, \mc O_Y),\\
		V_2 &\cong H^0(Y, \Omega_Y(R')) \oplus 
		H^1(Y, \mc O_Y( - D))^{\chi^{-1}}
		\ominus H^1(Y, \mc O_Y)^{\chi^{-1}} \ominus H^0(Y, \Omega_Y),
	\end{align*}
	 the divisor $R'$ is given by~\eqref{eqn:R'_definition} and $D \in \Divv(Y)$ is defined as:
	\begin{align*}
		D &:= \sum_{Q \in B_{X/Y}} \left\lceil \frac{u_{X/Y, Q}}{p} \right\rceil \cdot Q.
	\end{align*}
\end{Theorem}
Here $W := W_1 \ominus W_2$ means that $W_1 \cong W \oplus W_2$ (note that if such a $W$ exists, it is unique). By Proposition~\ref{prop:chevalley_weil}, the representations defining $V_1$ and $V_2$ are determined by the ramification data.\\

We sketch now briefly the idea behind the proof of Theorem~\ref{thm:Zp_formula}. Using techniques similar as in the proof of Main Theorem, one concludes that the $k[G]$-module structure of $H^1_{dR}(X)$ is determined by the $k[C]$-structure of $T^1 H^1_{dR}(X)$
and $T^p H^1_{dR}(X)$. Moreover, $T^p H^1_{dR}(X) \cong H^1_{dR}(Y)$. In order to determine $H^1_{dR}(X)^H$ consider the sequence of invariants arising from the Hodge--de Rham exact sequence~\eqref{eqn:hdr_exact_sequence}:
\begin{equation} \label{eqn:hdr_exact_sequence_invariants}
	0 \to H^0(X, \Omega_X)^H \to H^1_{dR}(X)^H \to H^1(X, \mc O_X)^H.
\end{equation}
We compute explicitly the image $\II_{dR}$ of $H^1_{dR}(X)^H$ in $H^1(X, \mc O_X)^H$ (cf. Proposition~\ref{prop:image_of_HdR_invariants}~(2)).
To this end we need some auxiliary results.
For a sheaf with an action of $H$ we write again $\mc F^{(i)} := \ker( (\sigma - 1)^i : \mc F \to \mc F)$
and $T^i \mc F := \mc F^{(i)}/\mc F^{(i-1)}$. Since in the sequel we work with sheaves on~$Y$, we abbreviate $\pi_* \mc O_X$ to $\mc O_X$, $\pi_* \mc O_X^{(i)}$ to $\mc O_X^{(i)}$ etc. Note that $H^i(X, \mc F) = H^i(Y, \pi_* \mc F)$ for any $\mc O_X$-module $\mc F$ and any $i \ge 0$, since $\pi$ is an affine morphism.
We say that a function $y \in k(X)$ is an {\em Artin--Schreier generator in standard form on $U$},
if  $\sigma(y) = y+1$ and the poles of $f := y^p - y \in k(Y)$ contained in $U$ have order non-divisible by $p$.
Note that for any affine open subset $U \subset Y$ there exists an Artin--Schreier generator in standard form on $U$.  Moreover, if $y$ is an Artin--Schreier generator, then any other generator
is of the form $y + g$ for $g \in \mc O_Y(U)$.
\begin{Lemma} \label{lem:OX=bigoplus_O(Di)}
	Assume that $y$ is an Artin--Schreier generator in standard form on an open subset $U \subset Y$.
	Then:
	\begin{equation*} 
		\mc O_X|_U = \bigoplus_{i = 0}^{p-1} y^i \mc O_U(-D_i),
	\end{equation*}
	where $D_i := \sum_{Q \in B_{X/Y}} \left\lceil \frac{u_{X/Y, Q} \cdot i}{p} \right\rceil
	\cdot Q	\in \Divv(Y)$.
	In particular, for any $j = 0, \ldots, p{-}1$ we have the following isomorphism of $\mc O_Y$-modules:
	\begin{equation} \label{eqn:Tj_OX}
		T^j \mc O_X \cong \mc O_Y(-D_{j  - 1}).
	\end{equation}
\end{Lemma}
\begin{proof}
	By \cite[(7.8)]{Garnek_p_gp_covers} for any $g = g_0 + g_1  \cdot y + \cdots + g_{p-1} \cdot y^{p-1} \in k(X)$ and $P \in U$ we have $\ord_P(g) = \min \{\ord_P(g_i \cdot y^i) : i =0, \ldots, p-1\}$.
	Hence for $Q \in B_{X/Y}$ and $P \in \pi^{-1}(Q)$ we have $\ord_P(g) \ge 0$ if and only if $\ord_Q(g_i) \ge  \frac{u_{X/Y, P} \cdot i}{p}$. This proves the first equality. In order to 
	find $T^j \mc O_X$, note that
	\begin{equation}
		\label{eq:sheafFIX}
		\mc O_X^{(j)}|_U = \bigoplus_{i = 0}^{j-1} y^i \mc O_U(-D_i).
	\end{equation}
	This clearly implies that we have an isomorphism:
	\[
		T^j \mc O_X|_U \cong \mc O_Y(-D_{j  - 1})|_U, \qquad
		g = g_0 + y \cdot g_1 + \cdots + y^{ j-1} \cdot g_{ j-1} \mapsto g_{ j-1}.
	\]
	One easily checks that this isomorphism does not depend on the choice of $y$.
	Hence, by picking a cover of $Y$ by open affine subsets and choosing an Artin--Schreier generator in
	standard form on each of them, we obtain the desired isomorphism.
\end{proof}
\begin{Lemma} \label{lem:invariants_vs_coinvariants}
	Keep the above notation. Assume that $M$ is a $k[G]$-module and $M = M^{(2)}$. Then $M_H \oplus M^{\chi} \cong (M^H)^{\chi} \oplus M$ as $k[C]$-modules.
\end{Lemma}
\begin{proof}
	By a similar reasoning as in Lemma~\ref{lem:properties_TiM}~(1), $\sigma - 1$ is $\chi^{-1}$-linear. This yields the following short exact sequences:
	\begin{align*}
		0 \to M^H \to M \stackrel{\sigma - 1}{\longrightarrow} \im(\sigma - 1)^{\chi^{-1}} \to 0,\\
		0 \to \im(\sigma - 1) \to M \to M_H \to 0.
	\end{align*}
	Therefore, since the category of $k[C]$-modules is semisimple:
	\begin{align*}
		\im(\sigma - 1) \cong M^{\chi} \ominus (M^H)^{\chi} \cong M \ominus M_H.
	\end{align*}
	which yields $M_H \oplus M^{\chi} \cong (M^H)^{\chi} \oplus M$.
\end{proof}
\begin{Lemma} \label{lem:OX2_invariants}
	Keep the above notation. Then, as $\mc O_Y$-modules with an action of $C$:
	\[
	\big( (\mc O_X^{(2)})^{\vee} \big)^H \cong \mc O_Y(D)^{\chi},
	\]
	where $D$ is defined as in Theorem~\ref{thm:Zp_formula}.
\end{Lemma}
\begin{proof}
	The following maps are mutually inverse isomorphisms of $\mc O_Y$-modules:
	\begin{align*}
		\Psi : \mc O_Y(D) \to \Hom_{\mc O_Y}(\mc O_X^{(2)}, \mc O_Y)^H, \quad
		&\Psi(h) = \varphi_h,\\
		\Phi : \Hom_{\mc O_Y}(\mc O_X^{(2)}, \mc O_Y)^H \to \mc O_Y(D), \quad
		&\Phi(\varphi) = h_{\varphi},
	\end{align*}
	where for an open subset $U \subset Y$:
	\begin{align*}
		\varphi_h(g) &:= h \cdot (\sigma(g) - g) &&\textrm{ for } g \in H^0(U, \mc O_X^{(2)}) \textrm{ and } h \in H^0(U, \mc O_Y(D)),\\
		h_{\varphi} &:= \frac{\varphi(g)}{\sigma(g) - g} &&\textrm{ for } \varphi \in \Hom_{\mc O_U}(\mc O_X^{(2)}|_U, \mc O_U)^H\\
		&&&\textrm{ and any } g \in \mc O_X^{(2)}(U) \setminus \mc O_Y(U).
	\end{align*}
	In order to check that those maps are well-defined, pick an arbitrary affine open set
	$U \subset Y$ and an Artin--Schreier generator $y$ in standard form on~$U$.
	Then, by~\eqref{eq:sheafFIX} we have
	$(\mc O_X^{(2)})|_U = \mc O_U \oplus y \mc O_U(-D)$. 
	Hence, if $g = g_0 + g_1 \cdot y \in H^0(U, \mc O_X^{(2)})$ and $h \in H^0(U, \mc O_Y(D))$, then 
	$g_0 \in H^0(U, \mc O_Y)$, $g_1 \in H^0(U, \mc O_Y(-D))$ and $\varphi_h(g) = g_1 \cdot h \in H^0(U, \mc O_Y)$.
	One easily checks that $\varphi_h$ is $H$-invariant.
	
	We check now that if $\varphi \in H^0(U, \Hom_{\mc O_Y}(\mc O_X^{(2)}, \mc O_Y)^H)$
	then $h_{\varphi}$ is well-defined. Firstly, $h_{\varphi}$ does not depend on the choice of $g$. Indeed, if $g' = g_0' + y \cdot g_1'$ then:
	\begin{align*}
		\frac{\varphi(g)}{\sigma(g) - g} - \frac{\varphi(g')}{\sigma(g') - g'} &=
		\frac{\varphi(g_0 g_1' - g_0' g_1)}{g_1 g_1'} = 
		\frac{\varphi((\sigma - 1) \cdot g'')}{g_1 g_1'} = 0,
	\end{align*}
	where $g '' := (g_0 g_1' - g_0' g_1) \cdot y \in \mc O_X^{(2)}$, since $g_0 g_1' - g_0' g_1 \in H^0(U, \mc O_Y(-D))$.
	Secondly, pick $g_1 \in k(Y)$ in such a way that $(g_1)_{\infty}$ (the divisor of poles of $g$ on $U$) equals $D|_U$ (this is possible, since $U$ is affine).
	Then for $g := y \cdot g_1$ we clearly have $h_{\varphi} = \frac{\varphi(g)}{g_1} \in H^0(U, \mc O_Y(D))$. One easily checks that $\Psi$ and $\Phi$ are inverse to each other.
	Finally, we observe that when accounting for the action of $C$, the morphism $\Psi$ is $\chi^{-1}$-linear, which ends the proof.
\end{proof}
\begin{Lemma} \label{lem:H0(OX2_vee_omega)}
	As a $k[C]$-module:
	\[
	H^0(Y, (\mc O_X^{(2)})^{\vee} \otimes \Omega_Y) \cong H^0(Y, \Omega_Y) \oplus H^0(Y, \Omega_Y(D))^{\chi}.
	\]
\end{Lemma}
\begin{proof}
	Consider the following exact sequence:
	\[
		0 \to \mc O_Y \to \mc O_X^{(2)} \to T^2 \mc O_X^{(2)} \to 0.
	\]
	Note that $T^2 \mc O_X^{(2)} \cong T^2 \mc O_X \cong \mc O_Y(-D)$ as sheaves on $Y$ by~\eqref{eqn:Tj_OX}. Moreover, since the isomorphism is $\chi^{-1}$-linear,
	$T^2 \mc O_X^{(2)} \cong \mc O_Y(-D)^{\chi^{-1}}$ as sheaves with an action of $C$. Therefore, by considering the associated long exact sequence
	and noting that $H^0(Y, \mc O_Y(- D)) = 0$:
	\begin{align*}
		0 \to H^1(Y, \mc O_Y) \to H^1(Y, \mc O_X^{(2)}) \to H^1(Y, \mc O_Y(-D))^{\chi^{-1}} \to 0.
	\end{align*}
	Using Serre's duality and the fact that the category of $k[C]$-modules is semi-simple, we obtain
	\begin{align*}
		H^0(Y, (\mc O_X^{(2)})^{\vee} \otimes \Omega_Y) \cong H^1(Y, \mc O_X^{(2)})^{\vee}
		\cong H^0(Y, \Omega_Y) \oplus H^0(Y, \Omega_Y(D))^{\chi}.
	\end{align*} 
	
\end{proof}

\begin{Proposition} \label{prop:image_of_HdR_invariants}
	\begin{enumerate}
		\item[]
		\item There exists an isomorphism of $k[C]$-modules:
		\[
		H^1(X, \mc O_X^{(2)})^H \cong 
		H^1(Y, \mc O_Y) \oplus H^1(Y, \mc O_Y(-D))^{\chi^{-1}}
		\ominus H^1(Y, \mc O_Y)^{\chi^{-1}}.
		\]

		\item The module $\II_{dR}$ equals $\II := \im \left(H^1(X, \mc O_X^{(2)})^H \to H^1(X, \mc O_X)^H \right)$.
	\end{enumerate}
\end{Proposition}
\begin{proof}
	(1) Write $M := H^0(Y, (\mc O_X^{(2)})^{\vee} \otimes \Omega_Y)$. Then, by Serre's duality:
	\[
		H^1(X, \mc O_X^{(2)})^H \cong \left( M_H \right)^{\vee}
	\]
	Moreover, by Lemma~\ref{lem:OX2_invariants}:
	\begin{align}
		M^H &= H^0(Y, (\mc O_X^{(2)})^{\vee} \otimes \Omega_Y)^H  \nonumber \\ &\cong
		H^0(Y, \big( (\mc O_X^{(2)})^{\vee} \big)^H \otimes \Omega_Y) \nonumber\\
		&\cong H^0(Y, \Omega_Y(D))^{\chi}. \label{eqn:M^H_formula}
	\end{align}
	Therefore, using Lemma~\ref{lem:invariants_vs_coinvariants}, Lemma~\ref{lem:H0(OX2_vee_omega)} and \eqref{eqn:M^H_formula}:
	\begin{align*}
		M_H &\cong \left(H^0(Y, \Omega_Y) \oplus H^0(Y, \Omega_Y(D))^{\chi} \right)
		 \oplus H^0(Y, \Omega_Y(D))^{\chi^2}\\
		 &\ominus \left(H^0(Y, \Omega_Y) \oplus H^0(Y, \Omega_Y(D))^{\chi} \right)^{\chi}\\
		 &\cong H^0(Y, \Omega_Y) \oplus H^0(Y, \Omega_Y(D))^{\chi} \ominus H^0(Y, \Omega_Y)^{\chi}.
	\end{align*}
	This, along with Serre's duality, finishes the proof.\\
	(2) Recall that as a $k[H]$-module $H^1_{dR}(X)$ is a direct sum of copies of $J_p$ and $J_{p-1}$  by Theorem~\ref{thm:cyclic_de_rham}.
	This easily implies that
	\[
	H^1_{dR}(X)^H = \im( (\sigma - 1)^{p-2} : H^1_{dR}(X)^{(p-1)} \to H^1_{dR}(X)^{(p-1)}).
	\]
	Therefore the map $H^1_{dR}(X)^H \to H^1(X, \mc O_X)^H$ factors through the group
	\[
	L := \im( (\sigma - 1)^{p-2} : H^1(X, \mc O_X) \to H^1(X, \mc O_X)).
	\]
	Let $\mc I$ be the sheaf image of $(\sigma - 1)^{p-2} : \mc O_X \to \mc O_X$. Since $(\sigma - 1)^{p-2} : \mc O_X \to \mc O_X$ factors through $\mc I$,
	the map $L \to H^1(X, \mc O_X)$ factors through $H^1(X, \mc I)$. On the other hand, the inclusion $\mc I \subset \mc O_X^{(2)}$
	implies that $H^1(X, \mc I) \to H^1(X, \mc O_X)$ factors through $H^1(X, \mc O_X^{(2)})$. Hence $\II_{dR} \subset \II$. Note that the map
	$H^1(X, \mc O_X^{(2)}) \to H^1(X, \mc O_X)$ is injective. Indeed, 
	using~\eqref{eqn:Tj_OX} and induction one may easily prove that $H^0(X, \mc O_X/\mc O_X^{(j)}) = 0$
	for any $0 \le j \le p-1$. Hence:
	\[
		\ker \left(H^1(X, \mc O_X^{(2)}) \to H^1(X, \mc O_X) \right) = H^0(X, \mc O_X/\mc O_X^{(2)}) = 0.
	\]
	Recall that $\dim_k H^1_{dR}(X)^H = 2 g_Y + \alpha$ (cf. \eqref{eqn:dr_invariants_for_Zp})
	and $\dim_k H^0(X, \Omega_X)^H = g_Y - 1 + \deg R'$. Therefore, by part~(1):
	\begin{align*}
		\dim_k \II_{dR} &= \dim_k H^1_{dR}(X)^H - \dim_k H^0(X, \Omega_X)^H\\
		&= g_Y - 1 + \deg D = \dim_k H^0(Y, \Omega_Y(D)) = \dim_k \II.
	\end{align*}
	This shows that $\II_{dR} = \II$ and finishes the proof.
\end{proof}
\begin{proof}[Proof of Theorem~\ref{thm:Zp_formula}]
	By Theorem~\ref{thm:cyclic_de_rham} there exist $k[C]$-modules $V_1, V_2$ such that:
	\[
		H^1_{dR}(X) \cong J_p(V_1) \oplus J_{p-1}(V_2).
	\]
	By Proposition~\ref{prop:image_of_HdR_invariants} and~\eqref{eqn:hdr_exact_sequence_invariants}:
	\[
		0 \to H^0(X, \Omega_X)^H \to H^1_{dR}(X)^H \to \II_{dR} \to 0.
	\]
	Recall that $\pi_*^H \Omega_X \cong \Omega_Y(R')$ (cf. \cite[Corollary~2.4]{Garnek_equivariant}). Therefore, as $k[C]$-modules:
	\begin{align*}
		H^1_{dR}(X)^H &\cong H^0(X, \Omega_X)^H \oplus \II_{dR} \cong H^0(Y, \pi_*^H \Omega_X) \oplus \II_{dR}\\
		&\cong H^0(Y, \Omega_Y(R')) \oplus H^1(Y, \mc O_Y)\\
		&\oplus H^1(Y, \mc O_Y(-D))^{\chi^{-1}}
		\ominus H^1(Y, \mc O_Y)^{\chi^{-1}}.
	\end{align*}
	On the other hand, $H^1_{dR}(X)^H \cong V_1 \oplus V_2$, which yields a formula for $V_1 \oplus V_2$.
	Moreover, the map $\tr_{X/Y} : H^1_{dR}(X) \to H^1_{dR}(Y)$
	induces an isomorphism
	\[
		T^p H^1_{dR}(X) \cong H^1_{dR}(Y) \cong H^0(Y, \Omega_Y) \oplus  H^1(Y, \mc O_Y)
	\]
	(see~\eqref{eqn:trace_isomorphism}). One finishes the proof, by noting that
	\[
		T^p H^1_{dR}(X) \cong V_1^{\chi^{-p+1}} \cong V_1. \qedhere
	\]
\end{proof}

\section{An example -- a superelliptic curve with a metacyclic action} \label{sec:example}

Let $p > 2$ be a prime and let $m$ be a natural number such that $p \nmid m$. Let $k$ be an algebraically closed field of characteristic~ $p$.
Fix a primitive root of unity $\zeta \in \ol{\FF}_p^{\times}$ of order $m \cdot (p-1)$.
Note that $\zeta^m \in \FF_p$. Let $V \subset \ol{\FF}_p$ be a $\FF_p$-linear space of order $p^n$ containing $\FF_p$.
In this section we compute the equivariant structure of the de Rham cohomology for the superelliptic curve $X$ with the affine part given by:
\begin{equation*}
	y^m = f_V(x) := \prod_{v \in V}(x - v).
\end{equation*}
Note that for $m = 2$, $V = \FF_{p^2}$ this curve was considered e.g. in \cite[Section~4]{Bleher_Wood_polydiffs_structure}.
It is a curve of genus $\frac 12 (p^n - 1) (m-1)$ with an action of the group  $G' := V \rtimes_{\chi'} C$,
where $C := \langle \rho \rangle \cong \ZZ/(m \cdot (p - 1))$ and
\[
	\chi' : C \to \Aut(V), \qquad \chi'(\rho)(v) = \zeta^m \cdot v.
\]
This action is given by:
\begin{align*}
	\sigma_v(x, y) &= (x+v, y) \qquad \textrm{for } v \in V,\\
	\rho(x, y) &= (\zeta^{-m} \cdot x, \zeta^{-1} \cdot y).
\end{align*}
In the sequel we restrict our attention to a subgroup of $G'$ that satisfies the assumption
of Main Theorem. Namely, let $G := H \rtimes_{\chi} C$, where $H := \FF_p \subset V$ and (writing $\sigma$ for the generator of~$H$):
\[
	\chi : C \to \Aut(H), \quad \chi(\rho)(\sigma) = \sigma^{\zeta^m}.
\]
Let $\psi : C \to k^{\times}$ be the primitive character determined by $\psi(\rho) = \zeta$. 
Note that $\chi$ might be identified with $\psi^m$.
\begin{Proposition} \label{prop:superelliptic}
	Keep the above notation. There exists an isomorphism of $k[G]$-modules:
	\[
		H^1_{dR}(X) \cong J_p(V_1) \oplus J_{p-1}(V_2),
	\]
	where the $k[C]$-modules $V_1, V_2$ are as follows:
	\begin{align*}
		V_1 &\cong \bigoplus_{l = 0}^{m \cdot (p-1)} (\psi^l)^{\oplus (\alpha_l + \alpha_{ - l})}\\
		 V_2 &\cong \bigoplus_{l = 0}^{m \cdot (p-1)} (\psi^l)^{\oplus (\gamma_l + \beta_{- l - m} - \alpha_{- l + m} - \alpha_{l})}
	\end{align*}
	and:
	\begin{align*}
		\alpha_l &:= \delta_l + \left\langle \frac{- c_3 \cdot l}{m \cdot (p-1)} \right\rangle
		+ \llbracket l = 0 \rrbracket,\\
		\beta_l &:= \delta_l + \frac{\lceil m/p \rceil}{m \cdot (p-1)} + \left\langle \frac{- \lceil m/p \rceil  - c_3 \cdot l}{m \cdot (p-1)} \right\rangle,\\
		\gamma_l &:= \delta_l + \frac{\lfloor (m+1) \cdot (p-1)/p \rfloor}{m \cdot (p-1)}
		+ \left\langle \frac{- \lfloor (m+1) \cdot (p-1)/p \rfloor  - c_3 \cdot l}{m \cdot (p-1)} \right\rangle,\\
		\delta_l &:= -1 + \left\langle \frac{l}{m} \right\rangle \cdot \frac{p^{n - 1} - 1}{p-1} +
		\left\langle \frac{l}{m \cdot (p-1)} \right\rangle.
	\end{align*}
	Here $c_1$, $c_2$ are any integers such that $c_2 \cdot p^{n - 1} - c_1 \cdot m = 1$ and $c_3$ is the multiplicative inverse of $c_2 - m \cdot c_1$ modulo $m \cdot (p-1)$.
\end{Proposition}
%
\noindent Let $Y$ be the superelliptic curve given by the equation:
\[
y^m = f_{\ol V}(z) := \prod_{\ol v \in \ol V} (z - \ol v),
\]
where $\ol V := \{ v^p - v : v \in V \}$. Observe that the map $V \to \ol V$, $v \mapsto v^p - v$ yields an isomorphism
$V/\FF_p \cong \ol V$. The map
\[
	X \to Y, \quad (x, y) \mapsto (x^p - x, y)
\]
is $H$-equivariant and of degree $p$,
which implies that it is the quotient map $X \to X/H$. The action of $C$ on $Y$
is given by
\[
	\rho(z) = \zeta^{-m} \cdot z, \quad \rho(y) = \zeta^{-1} \cdot y.
\]
The following lemma is crucial in the proof of Proposition~\ref{prop:superelliptic}.
\begin{Lemma} \label{lem:superelliptic_cohomologies}
	Keep the above notation. Then, as $k[H]$-modules:
	\begin{align*}
		H^0(Y, \Omega_Y) &\cong \bigoplus_{i = 0}^{m \cdot (p-1) - 1} (\psi^i)^{\oplus \alpha_i},\\
		H^0(Y, \Omega_Y(D)) &\cong \bigoplus_{i = 0}^{m \cdot (p-1) - 1} (\psi^i)^{\oplus \beta_i},\\
		H^0(Y, \Omega_Y(R')) &\cong \bigoplus_{i = 0}^{m \cdot (p-1) - 1} (\psi^i)^{\oplus \gamma_i}.
	\end{align*}
\end{Lemma}
\begin{proof}
Similarly as above, one checks that the quotient map $Y \to Y/C \cong \PP^1$ is given by $(y, z) \mapsto z^{p-1}$.
The set of   ramification  points of $Y \to Y/C$ is given by $\{ Q_{\infty}, Q_0, Q_1, \ldots, Q_N \} \subset (Y/C)(k)$, where
$N := \frac{p^{n-1} - 1}{p - 1}$, $Q_0 = 0$, $Q_{\infty} = \infty$ and $Q_1, \ldots, Q_N$ are
the elements of the set
\[
\{ \ol v^{p-1} : \ol v \in \ol V \setminus \{ 0\}  \} \subset (Y/C)(k).
\]
One easily checks that:
\[
C_{Q_i} =
\begin{cases}
	C, & \textrm{ for } i = 0, \infty,\\
	C', & \textrm{ for } i = 1, \ldots, N,
\end{cases}
\]
where $C' := \langle \rho^{p-1} \rangle \cong \ZZ/m$.
The branch points of $\pi : Y \to Y/C$ are the points of $Y$ given as follows:
\begin{itemize}
	\item points $P_0$ and $P_{\infty}$ above $Q_0$ and $Q_{\infty}$ respectively,
	
	\item points $P_i^{(1)}, \ldots, P_i^{(p-1)}$ above $Q_i$ for $i = 1, \ldots, N$.
\end{itemize}
Moreover, $B_{X/Y} = \{ P_{\infty} \}$ and $u_{X/Y, P_{\infty}} = m$.
Note that $y$ is the uniformizer of points $P_i^{(j)}$. Thus, since
$\rho(y) = \zeta^{-1} \cdot y$, we see that for $0 \le i \le N$, $1 \le t \le e_{P_i^{(j)}} - 1$:
\begin{align*}
	\theta_{Y/\PP^1, Q_i} &  =
	\begin{cases}
		\psi^{-1}, & \textrm{ if } i = 0,\\
		(\psi')^{-1}, & \textrm{ otherwise,}
	\end{cases}
	\\
	N_{Q_i, t}(\psi^l)  &  = 
	\begin{cases}
		\llbracket -t \equiv l \pmod{m \cdot (p-1)} \rrbracket, & \textrm{ if } i = 0,\\
		\llbracket -t \equiv l \pmod{m} \rrbracket, & \textrm{ otherwise,}
	\end{cases}	
\end{align*}
where $\psi' := \psi|_{C'}$. Moreover, since the uniformizer at $P_{\infty}$ is given by $x^{c_1}/y^{c_2}$, we obtain 
\[
\theta_{ Y/\PP^1, Q_{\infty}} = \psi^{c_2- m \cdot c_1}, \quad
N_{ Q_{\infty}, t}(\psi^l)
= \llbracket (c_2 - m \cdot c_1) \cdot t \equiv l \pmod{m \cdot (p-1)} \rrbracket.
\]
Therefore, by Proposition~\ref{prop:chevalley_weil}, the multiplicity of the character $\psi^l$ in the $k[H]$-module $H^0(X, \Omega_X)$ equals:
\begin{align*}
	(g_{\PP^1} - 1) &+ \sum_{j = 1}^N \sum_{t = 1}^{m - 1} \left\langle \frac{-t}{m} \right\rangle \cdot N_{Q_j, t}(\psi^l) + \sum_{t = 1}^{(p-1) \cdot m - 1} \left\langle \frac{-t}{(p-1) \cdot m} \right\rangle \cdot N_{Q_0, t}(\psi^l)\\
	&+ \sum_{t = 1}^{(p-1) \cdot m - 1} \left\langle \frac{-t}{(p-1) \cdot m} \right\rangle \cdot N_{Q_{\infty}, t}(\psi^l) + \llbracket l = 0 \rrbracket.
\end{align*}
This expression is easily seen to be equal to $\alpha_l$.
The formulas for $H^0(Y, \Omega_Y(D))$ and $H^0(Y, \Omega_Y(R'))$ follow analogously from Proposition~\ref{prop:chevalley_weil}.	
\end{proof}
\begin{proof}[Proof of Proposition~\ref{prop:superelliptic}]
By Theorem~\ref{thm:Zp_formula} we have $H^1_{dR}(X) \cong J_p(V_1) \oplus J_{p-1}(V_2)$, where (using Lemma~\ref{lem:superelliptic_cohomologies}):
\begin{align*}
	V_1 &\cong H^0(Y, \Omega_Y) \oplus H^1(Y, \mc O_Y) \cong \bigoplus_{l = 0}^{m \cdot (p-1)} (\psi^l)^{\oplus (\alpha_l + \alpha_{ - l})}\\
	V_2 & \cong H^0(Y, \Omega_Y(R')) \oplus 
	H^1(Y, \mc O_Y( - D))^{\chi^{-1}}
	\ominus H^1(Y, \mc O_Y)^{\chi^{-1}} \ominus H^0(Y, \Omega_Y)\\
	&\cong \bigoplus_{l = 0}^{m \cdot (p-1)} (\psi^l)^{\oplus (\gamma_l + \beta_{- l - m} - \alpha_{- l + m} - \alpha_{l})}.
\end{align*}
The result follows.
\end{proof}
\def\cprime{$'$}

%
\end{document}